\documentclass[11pt,a4paper]{article}

\usepackage{amsmath} 
\usepackage{amssymb}
\usepackage{amsthm} 
\usepackage{enumerate}
\usepackage{hyperref}
\usepackage{mathtools} 

\usepackage{url}

\usepackage{cleveref} 

\theoremstyle{plain}
\newtheorem{theorem}{Theorem}[section]
\crefname{theorem}{theorem}{theorems}
\Crefname{theorem}{Theorem}{Theorems}

\theoremstyle{definition}
\newtheorem{definition}[theorem]{Definition}
\crefname{definition}{definition}{definitions}
\Crefname{definition}{Definition}{Definitions}

\theoremstyle{remark}
\newtheorem{remark}[theorem]{Remark}
\crefname{remark}{remark}{remarks}
\Crefname{remark}{Remark}{Remarks}

\newcommand{\diff}{\mathrm{d}}
\newcommand{\diffusion}{d}
\newcommand{\fulldim}{n}
\newcommand{\ham}{\mathcal{H}}
\newcommand{\minEig}{\lambda_{\mathrm{min}}}
\newcommand{\N}{\mathbb{N}}
\newcommand{\nSnapshots}{q}
\newcommand{\R}{\mathbb{R}}
\newcommand{\sol}{x}
\newcommand{\solTimeDiscrete}{\breve\sol}
\newcommand{\spaceVar}{\xi}
\newcommand{\tend}{t_{\mathrm{end}}}
\newcommand{\timeInterval}{\mathbb{T}}
\newcommand{\uTimeDiscrete}{\breve{u}}
\newcommand{\yTimeDiscrete}{\breve{y}}
\newcommand{\zph}{z}

\newcommand{\norm}[1]{\left\lVert #1\right\rVert}

\begin{document}
\title{Structure-Preserving Time Discretization of Port-Hamiltonian Systems via Discrete Gradient Pairs}
\author{Philipp Schulze\,\thanks{Technische Universit\"at Berlin, Institute of Mathematics, Berlin, Germany,
\texttt{pschulze@math.tu-berlin.de}.}}

\maketitle            

\begin{abstract}
We discuss structure-preserving time discretization for nonlinear port-Hamiltonian systems with state-dependent mass matrix.
Such systems occur, for instance, in the context of structure-preserving nonlinear model order reduction for port-Hamiltonian systems and, in this context, structure-preserving time discretization is crucial for preserving some of the properties of the time-continuous reduced-order model.
For this purpose, we introduce a new class of time discretization schemes which is based on so-called discrete gradient pairs and leads to an exact power balance on the time-discrete level.
Moreover, for the special case of a pointwise symmetric and positive definite mass matrix, we present an explicit construction of a discrete gradient pair.
Finally, we illustrate the theoretical findings by means of a numerical example, where the time-continuous system is a nonlinear reduced-order model for an advection--diffusion problem.
\vskip .3truecm

{\bf Keywords:} port-Hamiltonian systems, structure-preserving time discretization, discrete gradient methods
\vskip .3truecm

{\bf AMS(MOS) subject classification:} 35Q49, 65P10, 93C55

\end{abstract}

\section{Introduction}

The modeling of physical systems often leads to systems of ordinary differential equations (ODEs) or differential-algebraic equations with particular properties, such as stability or the satisfaction of conservation laws.
In general, such properties may be lost after discretization in time, which may lead to numerical results revealing unphysical behavior, see for instance \cite[ch.~I]{HaiLW06}.
A possible approach for avoiding such issues is to use a structure-preserving time discretization scheme, since the properties are often encoded in an algebraic or geometric structure of the original continuous-time system.
Examples include gradient \cite{HirS74}, Hamiltonian \cite{Arn89}, or port-Hamiltonian (pH) \cite{SchJ14} structures.

Structure-preserving time discretization for Hamiltonian systems has a long history, see for instance \cite{HaiLW06} for a general overview.
In \cite{KotL19} the authors discuss structure-preserving time discretization for pH systems and demonstrate, among other things, that certain classes of collocation methods lead to an exact power balance on the time-discrete level, provided that the Hamiltonian is a quadratic function of the state.
This result is extended to descriptor systems in \cite{MehM19}.
Exact time-discrete power balances may be also obtained for non-quadratic Hamiltonians when using discrete gradient approaches, cf.~\cite{CelH17,FroGLM23}.
Moreover, a structure-preserving time discretization scheme based on a Petrov--Galerkin projection is presented in \cite{EggHS21}.
Structure-preserving techniques for other structures have, e.g., been considered in \cite{JueST19,KunM23,Oet18}.

Those approaches mentioned above which yield an exact power balance on the discrete level for general Hamiltonians have in common that they consider time-continuous systems where the gradient of the Hamiltonian occurs explicitly in the system equations.
In contrast, we consider in the following nonlinear pH systems with state-dependent mass matrix as introduced in \cite{MehM19}.
The difficulty in applying, for instance, discrete gradient methods to such systems is that the gradient of the Hamiltonian does not explicitly occur in the system equations, but only implicitly, see \cref{sec:problemSetting} for more details.
The main contributions of this manuscript are listed in the following.

\begin{itemize}
	\item We introduce the notion of discrete gradient pairs and a corresponding class of time discretization schemes.
		Especially, we show that this class yields an exact power balance on the time-discrete level, cf.~\Cref{thm:timeDiscreteDissipationInequality}.
	\item For the special case where the mass matrix is pointwise symmetric and positive definite, we present an explicit construction of a discrete gradient pair, the so-called midpoint discrete gradient pair, cf.~\Cref{thm:midpointDiscreteGradientPair}.
	\item We demonstrate the satisfaction of the time-discrete power balance by means of a numerical example and show that for this example the experimental order of convergence of the midpoint discrete gradient pair approach is the same as for the implicit midpoint rule, cf.~\cref{sec:numerics}.
\end{itemize}

The remainder of the paper is structured as follows.
In the following section, we formulate the considered mathematical problem.
Then, in \cref{sec:discGrad} we summarize the main idea of discrete gradients and their application to the structure-preserving time discretization of Hamiltonian systems.
The main results are provided in \cref{sec:discGradPairs}, where we introduce discrete gradient pairs and demonstrate how we may use them to obtain a structure-preserving time discretization scheme for nonlinear port-Hamiltonian systems with state-dependent mass matrix.
These theoretical findings are illustrated by means of a numerical example in \cref{sec:numerics}, before we provide a summary and an outlook in \cref{sec:conclusion}.

\paragraph*{Notation}

The set of real numbers is denoted with $\R$ and we use $\R^{m, n}$ for the set of $m\times n$ matrices with real-valued entries.
Moreover, we use $A^\top$ for the transpose of a matrix $A$.
Furthermore, to indicate that a matrix $A\in\R^{m,m}$ is positive (semi-)definite, we use the notation $A>0$ ($A\ge 0$).
For column vectors, we abbreviate $\R^{m,1}$ as $\R^m$ and we write $\norm{\cdot}$ for the Euclidean norm on $\R^m$.
The spaces of continuous and continuously differentiable functions from a suitable subset $U\subseteq \R^m$ to $\R^n$ are denoted with $C(U,\R^n)$ and $C^1(U,\R^n)$, respectively.
Finally, for a function $f$ depending on multiple variables $x_1,\ldots,x_m$, we use the short-hand notation $\partial_{x_i}f\vcentcolon=\frac{\partial f}{\partial x_i}$ for the partial derivative of $f$ with respect to $x_i$ for $i\in\lbrace 1,\ldots,m\rbrace$.

\section{Problem Setting}\label{sec:problemSetting}

We consider  port-Hamiltonian systems of the form
\begin{subequations}
	\label{eq:nonlinearTimeInvariantPHDAE}
	\begin{align}
		\label{eq:PHDAE_stateEq}
		E(\sol(t))\dot{\sol}(t) 	&= (J(\sol(t))-R(\sol(t)))\zph(\sol(t))+B(\sol(t))u(t),\\
		y(t)  							&= B(\sol(t))^\top\zph(\sol(t)),
	\end{align}
\end{subequations}
for all $t\in\timeInterval=[0,\tend]$, with state $\sol\colon \timeInterval\to\R^{\fulldim}$, input $u\colon \timeInterval\to\R^m$, output $y\colon\timeInterval\to \R^m$, and coefficient functions $E,J,R\in C(\R^{\fulldim},\R^{\fulldim,\fulldim})$, $\zph\in C(\R^{\fulldim},\R^{\fulldim})$, and $B\in C(\R^{\fulldim},\R^{\fulldim,m})$.
Associated with \eqref{eq:nonlinearTimeInvariantPHDAE} we consider the Hamiltonian $\ham\colon C^1(\R^{\fulldim},\R)$ and require the coefficients to satisfy pointwise
\begin{equation}
	\label{eq:pHProperties}
	J=-J^\top,\quad R=R^\top\ge 0,\quad E^\top\zph = \nabla\ham,
\end{equation}
cf.~\cite{MehM19}.
As demonstrated in \cite{MehM19}, the properties \eqref{eq:pHProperties} imply that each solution $\sol\in C^1(\timeInterval,\R^{\fulldim})$ of \eqref{eq:PHDAE_stateEq} satisfies the dissipation inequality
\begin{equation*}
	\frac{\diff\ham(\sol(t))}{\diff t} = -\zph(\sol(t))^\top R(\sol(t))\zph(\sol(t))+y(t)^\top u(t) \le y(t)^\top u(t)
\end{equation*}
for all $t\in\timeInterval$.

The goal of this paper is to derive a one-step time integration scheme for \eqref{eq:nonlinearTimeInvariantPHDAE}, based on a time grid $0=t_1<t_2<\ldots<t_{\nSnapshots}=\tend$ and yielding time-discrete approximations $\solTimeDiscrete^1,\ldots,\solTimeDiscrete^{\nSnapshots}$ for $\sol(t_1),\ldots,\sol(t_{\nSnapshots})$, respectively, satisfying a time-discrete power balance of the form
\begin{equation*}
	\frac{\ham(\solTimeDiscrete^{k+1})-\ham(\solTimeDiscrete^k)}{t_{k+1}-t_k} = -f_1(\solTimeDiscrete^k,\solTimeDiscrete^{k+1})+f_2(\solTimeDiscrete^k,\solTimeDiscrete^{k+1},u(\tfrac{t_{k+1}+t_k}2))
\end{equation*}
for $k=1,\ldots,\nSnapshots-1$ with discrete dissipation function $f_1\in C(\R^{\fulldim}\times\R^{\fulldim},\R_{\ge 0})$ and discrete supply rate function $f_{2}\in C(\R^{\fulldim}\times\R^{\fulldim}\times \R^m,\R)$ with $f_{2}(\eta,\xi,0)=0$ for all $(\xi,\eta)\in\R^{\fulldim}\times\R^{\fulldim}$.
Moreover, we require for consistency
\begin{equation*}
	f_1(\sol,\sol) = \zph(\sol)^\top R(\sol)\zph(\sol),\quad f_{2}(\sol,\sol,\hat{u}) = \zph(\sol)^\top B(\sol)\hat{u}
\end{equation*}
for all $\sol\in\R^{\fulldim}$ and $\hat{u}\in \R^m$.

\section{Time Discretization based on Discrete Gradients}\label{sec:discGrad}

For $\ham\in C^1(\R^{\fulldim})$, we call $\overline{\nabla} \ham\in C(\R^{\fulldim}\times\R^{\fulldim},\R^{\fulldim})$ a discrete gradient of $\ham$ if
\begin{enumerate}[(i)]
	\item $\overline{\nabla} \ham(\sol,\sol) = \nabla \ham(\sol)$ holds for all $\sol\in\R^n$, and
	\item $\overline{\nabla} \ham(\sol,\hat{\sol})^T(\hat{\sol}-\sol) = \ham(\hat{\sol})-\ham(\sol)$ hols for all $(\sol,\hat{\sol})\in\R^{\fulldim}\times\R^{\fulldim}$.
\end{enumerate}
An example for a discrete gradient is given by the \emph{midpoint discrete gradient}
\begin{equation}
	\label{eq:midpointDiscreteGradient}
	\overline{\nabla} \ham(\sol,\hat\sol) \vcentcolon= 
	\begin{cases}
		\nabla\ham\left(\frac{\hat\sol+\sol}2\right)+\frac{\ham(\hat\sol)-\ham(\sol)-\nabla\ham\left(\frac{\hat\sol+\sol}2\right)^\top(\hat\sol-\sol)}{\norm{\hat\sol-\sol}^2}(\hat\sol-\sol), & \text{if }\sol\ne\hat\sol,\\
		\nabla\ham(\sol), & \text{otherwise},
	\end{cases}
\end{equation}
cf.~\cite{Gon96}.
Discrete gradients are especially useful in the context of structure-preserving time discretization of Hamiltonian systems of the form
\begin{equation}
	\label{eq:HamSys}
	\dot{\sol}(t) = J\nabla \ham(\sol(t))\quad \text{for all } t\in\timeInterval,\quad \sol(0) = \sol_0
\end{equation}
with $J=-J^\top\in\R^{\fulldim,\fulldim}$.
Here, the structure implies the conservation of the Hamiltonian, which follows from the computation
\begin{equation*}
	\frac{\diff \ham(\sol(t))}{\diff t}(t) = \nabla\ham(\sol(t))^\top\dot{\sol}(t) = \nabla\ham(\sol(t))^\top J\nabla \ham(\sol(t)) = 0
\end{equation*}
for all $t\in\timeInterval$.
For the time discretization, we consider a time grid $0=t_1<t_2<\ldots<t_{\nSnapshots}=\tend$ and the time-discrete system
\begin{align}
	\label{eq:discreteSystem}
	\begin{aligned}
		\solTimeDiscrete^{k+1} &= \solTimeDiscrete^k+(t_{k+1}-t_k)J\overline{\nabla} H\left(\solTimeDiscrete^k,\solTimeDiscrete^{k+1}\right)\quad \text{for }k=1,\ldots,\nSnapshots-1,\\
		\solTimeDiscrete^{1} 																&= \sol_0.
	\end{aligned}
\end{align}
Based on the defining properties of the discrete gradient $\overline{\nabla}H$, we obtain
\begin{align*}
	\ham\left(\solTimeDiscrete^{k+1}\right)-\ham\left(\solTimeDiscrete^k\right) &= \overline{\nabla} \ham\left(\solTimeDiscrete^k,\solTimeDiscrete^{k+1}\right)^\top\left(\solTimeDiscrete^{k+1}-\solTimeDiscrete^k\right)\\
	&=(t_{k+1}-t_k)\overline{\nabla} \ham\left(\solTimeDiscrete^k,\solTimeDiscrete^{k+1}\right)^\top J\overline{\nabla} \ham\left(\solTimeDiscrete^k,\solTimeDiscrete^{k+1}\right)^\top = 0
\end{align*}
for $k=1,\ldots,\nSnapshots-1$, i.e., the Hamiltonian is also a conserved quantity of the time-discrete system.

In \cite{McLQR99} it is shown that discrete gradients may also be useful in the context of dissipative Hamiltonian systems, where the Hamiltonian is not conserved but is non-increasing in time instead.
Furthermore, applying a discrete gradient scheme to a pH system \eqref{eq:nonlinearTimeInvariantPHDAE} with $E=I_n$ yields that the solution of the time-discrete system satisfies a time-discrete analogue of the power balance, cf.~\cite{CelH17,FroGLM23}.

\section{Time Discretization based on Discrete Gradient Pairs}\label{sec:discGradPairs}

The major challenge in extending discrete gradient schemes to systems of the form \eqref{eq:nonlinearTimeInvariantPHDAE} is that the gradient of the Hamiltonian does not occur explicitly, but only implicitly due to \eqref{eq:pHProperties}.
Therefore, we introduce the notion of discrete gradient pairs in the following.

\begin{definition}[Discrete gradient pair]
	\label{def:discGradPair}
	Let $\ham\in C^1(\R^{\fulldim})$ with $\fulldim\in\N$ be given and let $E\in C(\R^{\fulldim},\R^{\fulldim,\fulldim})$ and $\zph\in C(\R^{\fulldim},\R^{\fulldim})$ satisfy
	\begin{equation}
		\label{eq:factorizationOfGradient}
		\nabla\ham(\sol) = E(\sol)^\top\zph(\sol)\quad \text{for all }\sol\in\R^{\fulldim}.
	\end{equation}
	Then, we call $(\overline{E},\overline{\zph})\in C(\R^{\fulldim}\times \R^{\fulldim},\R^{\fulldim,\fulldim})\times C(\R^{\fulldim}\times \R^{\fulldim},\R^{\fulldim})$ a \emph{discrete gradient pair} for $(\ham,E,\zph)$ if the following conditions are satisfied.
	\begin{enumerate}[(i)]
		\item \label{itm:discGradPH2}$\overline{E}(\sol,\sol) = E(\sol)$ for all $\sol\in\R^{\fulldim}$,
		\item \label{itm:discGradPH3}$\overline{\zph}(\sol,\sol) = \zph(\sol)$ for all $\sol\in\R^{\fulldim}$,
		\item \label{itm:discGradPH4}$\overline{\zph}(\sol,\hat\sol)^\top\overline{E}(\sol,\hat\sol)(\hat\sol-\sol) = \ham(\hat\sol)-\ham(\sol)$ for all $(\hat\sol,\sol)\in\R^{\fulldim}\times \R^{\fulldim}$.
	\end{enumerate}
\end{definition}

In the following we demonstrate that in the special case where $E$ is pointwise symmetric and positive definite we may explicitly construct a discrete gradient pair in a similar way as the midpoint discrete gradient considered in the previous section.

\begin{theorem}[Midpoint discrete gradient pair]
	\label{thm:midpointDiscreteGradientPair}
	Let $\ham\in C^1(\R^{\fulldim})$, $E\in C(\R^{\fulldim},\R^{\fulldim,\fulldim})$, and $\zph\in C(\R^{\fulldim},\R^{\fulldim})$ satisfy pointwise the last equality in \eqref{eq:pHProperties}.
	Furthermore, let $E$ be pointwise symmetric and positive definite.
	Then, a discrete gradient pair for $(\ham,E,\zph)$ is given by $\overline{E}\colon \R^{\fulldim}\times \R^{\fulldim}\to \R^{\fulldim,\fulldim}$ and $\overline{\zph}\colon \R^{\fulldim}\times \R^{\fulldim}\to \R^{\fulldim}$ defined via
	\begin{align}
		\overline{E}(\sol,\hat\sol) &\vcentcolon= E\left(\tfrac{\hat\sol+\sol}2\right),\\
		\label{eq:discGrad_zph}
		\overline{\zph}(\sol,\hat\sol) &\vcentcolon= 
		\begin{cases}
			\zph\left(\tfrac{\hat\sol+\sol}2\right)+\frac{\ham(\hat\sol)-\ham(\sol)-\zph\left(\tfrac{\hat\sol+\sol}2\right)^\top\overline{E}(\sol,\hat\sol)(\hat\sol-\sol)}{(\hat\sol-\sol)^\top\overline{E}(\sol,\hat\sol)(\hat\sol-\sol)}(\hat\sol-\sol), & \text{if }\hat\sol\ne\sol,\\
			\zph(\sol), & \text{otherwise}.
		\end{cases}
	\end{align}
\end{theorem}

\begin{proof}
	The definitions of $\overline{E}$ and $\overline{z}$ imply that the conditions \eqref{itm:discGradPH2} and \eqref{itm:discGradPH3} from \Cref{def:discGradPair} are satisfied.
	Furthermore, condition \eqref{itm:discGradPH4} follows from a straightforward calculation exploiting the special construction of $\overline{\zph}$.	
	In addition, the continuity of $E$, $\zph$, $\ham$ and the pointwise symmetry and positive definiteness of $E$ imply that $\overline{E}$ is continuous and that $\overline{\zph}$ is continuous in $\lbrace (\hat\sol,\sol)\in\R^{\fulldim}\times\R^{\fulldim} \mid \hat\sol\ne \sol \rbrace$.
	It remains to show the continuity of $\overline{\zph}$ in $\lbrace (\hat\sol,\sol)\in\R^{\fulldim}\times\R^{\fulldim} \mid \hat\sol= \sol \rbrace$.
	To this end, let $\tilde\sol\in\R^{\fulldim}$ be arbitrary and let $(x_k,y_k)_{k\in\N}$ be a sequence in $\R^{\fulldim}\times\R^{\fulldim}$ with $(x_k,y_k)\ne (\tilde\sol,\tilde\sol)$ for all $k\in\N$ and $\lim_{k\to\infty}(x_k,y_k) = (\tilde\sol,\tilde\sol)$.
	If $x_k = y_k$ holds for all $k\in\N$, we have
	\begin{equation*}
		\lim_{k\to\infty}\norm{\overline{\zph}(x_k,y_k)-\overline{\zph}(\tilde\sol,\tilde\sol)} = \lim_{k\to\infty}\norm{\zph(x_k)-\zph(\tilde\sol)} = 0.
	\end{equation*}		
	On the other hand, if $x_k\ne y_k$ holds for all $k\in\N$, Taylor's theorem yields
	\begin{align*}
		0 	&\le \lim_{k\to\infty}\norm{\overline{\zph}(x_k,y_k)-\overline{\zph}(\tilde\sol,\tilde\sol)}\\
			&\le \lim_{k\to\infty}\norm{\zph\left(\frac{x_k+y_k}2\right)-\zph(\tilde\sol)}\\
			&\quad+\lim_{k\to\infty}\left\lvert\frac{\ham(y_k)-\ham(x_k)-\nabla\ham\left(\frac{x_k+y_k}2\right)^\top(y_k-x_k)}{(y_k-x_k)^\top E\left(\frac{x_k+y_k}2\right)(y_k-x_k)}\right\rvert\norm{y_k-x_k}\\
			&\le \lim_{k\to\infty}\frac{\max_{t\in[-1,1]}\norm{\nabla\ham\left(\frac{x_k+y_k+t(y_k-x_k)}2\right)-\nabla\ham\left(\frac{x_k+y_k}2\right)}}{\minEig\left(E\left(\frac{x_k+y_k}2\right)\right)} = 0
	\end{align*}
	and, thus, $\lim_{k\to\infty}\norm{\overline{\zph}(x_k,y_k)-\overline{\zph}(\tilde\sol,\tilde\sol)} = 0$.
	All other cases may be reduced to the two considered ones by removing a finite number of sequence members or by splitting the sequence into two partial sequences.
\end{proof}

Similarly as in the previous section, we aim to use the concept of discrete gradient pairs to derive a suitable time-discrete approximation of \eqref{eq:nonlinearTimeInvariantPHDAE} which ensures a dissipation inequality on the time-discrete level.
To this end, we consider a time grid $0=t_1<t_2<\ldots<t_{\nSnapshots}=\tend$ and propose the discrete-time system
\begin{subequations}
	\label{eq:nonlinearTimeInvariantPHDAE_timeDiscrete}
	\begin{align}
		\label{eq:nonlinearTimeInvariantPHDAE_timeDiscretStateEq}
		\begin{split}
			\overline{E}(\solTimeDiscrete^k,\solTimeDiscrete^{k+1})\solTimeDiscrete^{k+1} &= \overline{E}(\solTimeDiscrete^k,\solTimeDiscrete^{k+1})\solTimeDiscrete^k+(t_{k+1}-t_k)\overline{B}(\solTimeDiscrete^k,\solTimeDiscrete^{k+1})\uTimeDiscrete^{k+\frac12}\\
			&+(t_{k+1}-t_k)\left(\overline{J}(\solTimeDiscrete^k,\solTimeDiscrete^{k+1})-\overline{R}(\solTimeDiscrete^k,\solTimeDiscrete^{k+1})\right)\overline{\zph}(\solTimeDiscrete^k,\solTimeDiscrete^{k+1}),
		\end{split}
		\\
		\label{eq:nonlinearTimeInvariantPHDAE_timeDiscreteOutput}
		\yTimeDiscrete^{k+\frac12}	 &= \overline{B}(\solTimeDiscrete^k,\solTimeDiscrete^{k+1})^\top\overline{\zph}(\solTimeDiscrete^k,\solTimeDiscrete^{k+1})
	\end{align}
\end{subequations}
for $k=1,\ldots,\nSnapshots-1$, where $\overline{E}$ and $\overline{\zph}$ are assumed to form a discrete gradient pair for $(\ham,E,\zph)$.
Furthermore, $\solTimeDiscrete^i\in\R^{\fulldim}$ corresponds to an approximation of $\sol(t_i)$ for $i=1,\ldots,\nSnapshots$.
The time-discrete input values are chosen as the evaluations of $u$ at the corresponding midpoint, i.e., $\uTimeDiscrete^{i+\frac12} = u(\frac12(t_i+t_{i+1}))$ for $i=1,\ldots,\nSnapshots-1$.
Moreover, we require that $\overline{B}\colon \R^{\fulldim}\times\R^{\fulldim}\to \R^{\fulldim,m}$ and $\overline{J},\overline{R}\colon \R^{\fulldim}\times\R^{\fulldim}\to \R^{\fulldim,\fulldim}$ are continuous and satisfy
\begin{equation*}
	\overline{B}(\sol,\sol) = B(\sol),\quad \overline{J}(\sol,\sol) = J(\sol),\quad \overline{R}(\sol,\sol) = R(\sol)
\end{equation*}
for all $\sol\in\R^{\fulldim}$.
Moreover, we assume that $\overline{J}$ is pointwise skew-symmetric and $\overline{R}$ pointwise symmetric and positive semi-definite.
All these properties are, for instance, satisfied when using the implicit midpoint rule for the approximation of $J$, $R$, and $B$.

By construction, the time-discrete system \eqref{eq:nonlinearTimeInvariantPHDAE_timeDiscrete} leads to a time-discrete analogue of the dissipation inequality as detailed in the following theorem.
We note in particular that $E$ is not required to be pointwise invertible, symmetric, or positive semi-definite.
Moreover, we emphasize that the terms occurring in the time-discrete power balance satisfy the properties mentioned in \cref{sec:problemSetting}.

\begin{theorem}[Dissipation inequality for the time-discrete system \eqref{eq:nonlinearTimeInvariantPHDAE_timeDiscrete}]
	\label{thm:timeDiscreteDissipationInequality}
	Consider a port-Hamiltonian system of the form \eqref{eq:nonlinearTimeInvariantPHDAE} with time interval $\timeInterval=[0,\tend]$, $\tend\in\R_{>0}$, associated Hamiltonian $\ham\in C^1(\R^{\fulldim})$, and coefficient functions $E,J,R\in C(\R^{\fulldim},\R^{\fulldim,\fulldim})$, $\zph\in C(\R^\fulldim,\R^\fulldim)$, and $B\in C(\R^\fulldim,\R^{\fulldim,m})$ satisfying pointwise \eqref{eq:pHProperties}.
	Furthermore, let $t_1,t_2,\ldots,t_{\nSnapshots}\in\timeInterval$ with $0=t_1<t_2<\ldots<t_{\nSnapshots}=\tend$ be given and let $(\overline{E},\overline{\zph})\in C(\R^{\fulldim}\times \R^{\fulldim},\R^{\fulldim,\fulldim})\times C(\R^{\fulldim}\times \R^{\fulldim},\R^{\fulldim})$ be a discrete gradient pair for $(\ham,E,\zph)$.
	Besides, let $\uTimeDiscrete^{\frac32},\uTimeDiscrete^{\frac52},\ldots,\uTimeDiscrete^{\nSnapshots-\frac12} \in\R^m$ be such that there exists a sequence $(\solTimeDiscrete^1,\solTimeDiscrete^2,\ldots,\solTimeDiscrete^{\nSnapshots})$ in $\R^{\fulldim}$ satisfying \eqref{eq:nonlinearTimeInvariantPHDAE_timeDiscretStateEq} for $k=1,\ldots,\nSnapshots-1$.
	Then, every such sequence satisfies the time-discrete dissipation inequality
	\begin{equation}
		\label{eq:timeDiscreteDissipationInequality}
		\begin{aligned}
			\frac{\ham(\solTimeDiscrete^{k+1})-\ham(\solTimeDiscrete^{k})}{t_{k+1}-t_k} &= -\overline{\zph}(\solTimeDiscrete^k,\solTimeDiscrete^{k+1})^\top \overline{R}(\solTimeDiscrete^k,\solTimeDiscrete^{k+1})\overline{\zph}(\solTimeDiscrete^k,\solTimeDiscrete^{k+1})+\left(\yTimeDiscrete^{k+\frac12}\right)^\top\uTimeDiscrete^{k+\frac12}\\
			&\le \left(\yTimeDiscrete^{k+\frac12}\right)^\top\uTimeDiscrete^{k+\frac12}
		\end{aligned}
	\end{equation}
	for $k=1,\ldots,\nSnapshots-1$, where $\yTimeDiscrete^{\frac32},\yTimeDiscrete^{\frac52},\ldots,\yTimeDiscrete^{\nSnapshots-\frac12}$ are defined via \eqref{eq:nonlinearTimeInvariantPHDAE_timeDiscreteOutput}.
\end{theorem}

\begin{proof}
	By exploiting \eqref{eq:pHProperties}, \eqref{eq:nonlinearTimeInvariantPHDAE_timeDiscretStateEq}, and the fact that $(\overline{E},\overline{\zph})$ is a discrete gradient pair for $(\ham,E,\zph)$, we obtain
	\begin{align*}
		&\frac{\ham(\solTimeDiscrete^{k+1})-\ham(\solTimeDiscrete^{k})}{t_{k+1}-t_k} = \frac{\overline{\zph}(\solTimeDiscrete^k,\solTimeDiscrete^{k+1})^\top\overline{E}(\solTimeDiscrete^k,\solTimeDiscrete^{k+1})(\solTimeDiscrete^{k+1}-\solTimeDiscrete^{k})}{t_{k+1}-t_k}\\
		&= \overline{\zph}(\solTimeDiscrete^k,\solTimeDiscrete^{k+1})^\top\left(\overline{J}(\solTimeDiscrete^k,\solTimeDiscrete^{k+1})-\overline{R}(\solTimeDiscrete^k,\solTimeDiscrete^{k+1})\right)\overline{\zph}(\solTimeDiscrete^k,\solTimeDiscrete^{k+1})\\
		&\quad+\overline{\zph}(\solTimeDiscrete^k,\solTimeDiscrete^{k+1})^\top\overline{B}(\solTimeDiscrete^k,\solTimeDiscrete^{k+1})\uTimeDiscrete^{k+\frac12}\\
		&= -\overline{\zph}(\solTimeDiscrete^k,\solTimeDiscrete^{k+1})^\top \overline{R}(\solTimeDiscrete^k,\solTimeDiscrete^{k+1})\overline{\zph}(\solTimeDiscrete^k,\solTimeDiscrete^{k+1})+\left(\yTimeDiscrete^{k+\frac12}\right)^\top\uTimeDiscrete^{k+\frac12}\\
		&\le \left(\yTimeDiscrete^{k+\frac12}\right)^\top\uTimeDiscrete^{k+\frac12}
	\end{align*}
	for $k=1,\ldots,\nSnapshots-1$.
\end{proof}

We note that \Cref{thm:timeDiscreteDissipationInequality} addresses only the power balance of the time-discrete system, whereas an analysis of the consistency and the convergence of the time discretization scheme is not within the scope of this paper.
Instead, we investigate the order of convergence numerically in the next section, where we use the midpoint discrete gradient pair from \Cref{thm:midpointDiscreteGradientPair}.

\begin{remark}[Discretization of \eqref{eq:nonlinearTimeInvariantPHDAE} by a classical discrete gradient method]
	If $E$ is pointwise invertible, \eqref{eq:nonlinearTimeInvariantPHDAE} may be transformed to the equivalent system
	\begin{equation}
		\label{eq:transformedPHDAE}
		\begin{aligned}
			\dot{\sol}(t) 	&= (\tilde J(\sol(t))-\tilde R(\sol(t))\nabla\ham(\sol(t))+\tilde B(\sol(t))u(t),\\
			y(t)  				&= \tilde B(\sol(t))^\top\nabla\ham(\sol(t))
		\end{aligned}
	\end{equation}
	for all $t\in\timeInterval$, with $\tilde J \vcentcolon= E^{-1}JE^{-\top}$, $\tilde R \vcentcolon= E^{-1}RE^{-\top}$, and $\tilde B\vcentcolon= E^{-1}B$.
	This transformed system is of the classical port-Hamiltonian ODE form.
	In particular, since the gradient of the Hamiltonian appears explicitly in \eqref{eq:transformedPHDAE}, the transformed system may be treated by classical discrete gradient methods as considered in the previous section.
	However, the notion of discrete gradient pairs as introduced in \Cref{def:discGradPair} allows to obtain a time discretization scheme as in \eqref{eq:nonlinearTimeInvariantPHDAE_timeDiscrete} without having to compute the inverse of $E$.
	Furthermore, the time-discrete dissipation inequality in \Cref{thm:timeDiscreteDissipationInequality} is also valid in the general case where $E$ may be singular.
\end{remark}

\section{Numerical Example}\label{sec:numerics}

As numerical test case we consider the linear advection--diffusion equation with mixed Robin--Neumann boundary conditions
\begin{equation}
	\label{eq:ADE}
	\left\{\begin{aligned}
		\partial_t\sol(t,\spaceVar) &= -c \partial_\spaceVar\sol(t,\spaceVar)+\diffusion\partial_{\spaceVar\spaceVar} \sol \left(t,\spaceVar\right) && \text{for all }(t,\spaceVar)\in\timeInterval\times\Omega,\\
		c\sol(t,0)-\diffusion\partial_{\spaceVar}\sol(t,0) &= cg(t) && \text{for all }t\in\timeInterval,\\
		\partial_\spaceVar \sol(t,1) &= 0 && \text{for all }t\in\timeInterval,\\
		\sol(0,\spaceVar) &= \sol_0(\spaceVar) && \text{for all }\spaceVar\in\Omega
	\end{aligned}\right.
\end{equation}
on the spatial domain $\Omega = (0,1)$.
For the parameters $c\in\R$ and $d\in\R_{\ge 0}$ and the time interval $\timeInterval$, we choose the same values as in \cite[sec.~5.1]{Sch23a}.

We also follow the spatial discretization and model order reduction as in \cite{Sch23a} and obtain a pH system of the form \eqref{eq:nonlinearTimeInvariantPHDAE} with pointwise symmetric and positive semi-definite $E$, cf.~\cite{Sch23a} for more details.
Moreover, for the considered solution trajectories, we have observed in our experiments that $E(\sol(t))$ is even positive definite for all $t\in\timeInterval$, which allows us to use the midpoint discrete gradient pair approach outlined in the previous section for the time discretization.

In \Cref{fig:ADE_ROM_powerBalanceError} we depict the mismatch in the discrete power balance for the implicit midpoint rule and the discrete gradient pair approach.
Since the Hamiltonian of the ROM is not a quadratic function of the ROM state, cf.~\cite{Sch23a}, the implicit midpoint rule yields a comparably large error in the power balance.
In contrast, for the discrete gradient pair approach, the theoretical results in the previous section yield that the corresponding time-discrete power balance is satisfied exactly, at least when ignoring the error of the nonlinear system solve in each time step.
In practice, this is reflected in a power balance error which is several orders of magnitude smaller than for the implicit midpoint rule. 

\begin{figure}[h!]
	\begin{center}
		\includegraphics[scale=1]{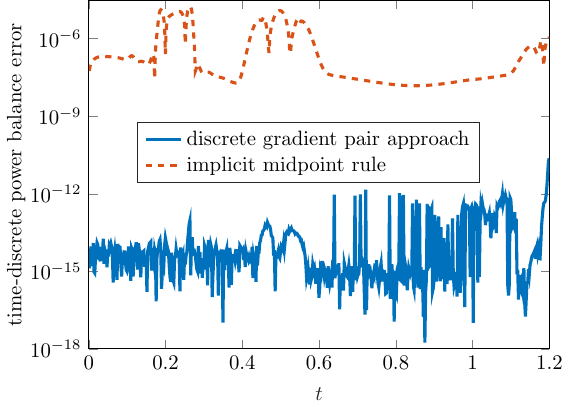}
	\end{center}
	\caption{Comparison of the discrete gradient pair approach and the implicit midpoint rule with time step size $10^{-3}$ in terms of the error in the time-discrete power balance.}
	\label{fig:ADE_ROM_powerBalanceError}
\end{figure}

While the fact that discrete gradient pair methods lead to an exact power balance on the time-discrete level is proven in \cref{sec:discGradPairs} and illustrated in \Cref{fig:ADE_ROM_powerBalanceError}, we have not yet addressed its convergence behavior.
To study the order of convergence numerically, we consider a reference solution obtained by solving the reduced-order model (ROM) via the RADAU IIA method of order five, cf.~\cite[p.~72ff.]{HaiW96}, with time step size $2\cdot 10^{-6}$.
Furthermore, to diminish the influence of the accuracy of the nonlinear equation system solver \texttt{fsolve}, we set the tolerances \texttt{OptimalityTolerance} and \texttt{FunctionTolerance} to $10^{-13}$ and $10^{-8}$, respectively.

Based on the reference solution, we determine the relative errors of the solutions obtained via the implicit midpoint rule and the discrete gradient pair approach for time step sizes ranging from $2\cdot 10^{-6}$ to $2^{17}\cdot 10^{-6}\approx 0.13$.
The specified error values correspond to the relative error with respect to the Frobenius norm of the solution snapshot matrices.
The error decays are depicted in \Cref{fig:discreteGradientVsImplMidpoint_convergence} together with a reference line for the convergence order two.
In particular, we observe that the convergence behavior of both methods is very similar and the discrete gradient pair approach is almost as accurate as the implicit midpoint rule.
In addition, the numerical results indicate a convergence order of two as it is to be expected for the implicit midpoint rule, see for instance \cite[sec.~6.3.2]{DeuB02}.

\begin{figure}[h!]
	\begin{center}
		\includegraphics[scale=1]{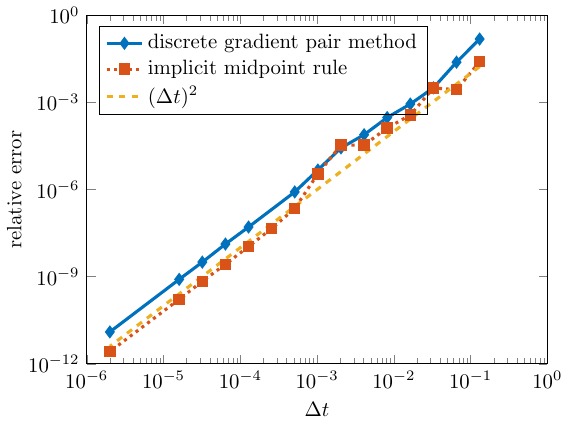}
	\end{center}
	\caption{Convergence of the implicit midpoint rule and the midpoint discrete gradient pair method.}
	\label{fig:discreteGradientVsImplMidpoint_convergence}
\end{figure}

\section{Conclusion}\label{sec:conclusion}

In this paper, we introduce a class of structure-preserving time discretization schemes for nonlinear port-Hamiltonian systems with state-dependent mass matrix.
To this end, we introduce the notion of discrete gradient pairs and show that they may be used to achieve an exact power balance on the time-discrete level.
We also explicitly construct a discrete gradient pair for the special case where the mass matrix is pointwise symmetric and positive definite.
The findings are supported by numerical experiments which involve the time discretization of a nonlinear reduced-order model for an advection--diffusion test case.

While we have presented an explicit construction of a discrete gradient pair in a special case, an interesting future research direction is the construction of further discrete gradient pairs with less restrictions on the mass matrix.
In this context we are especially interested in the case of nonlinear descriptor systems, where the mass matrix is singular.

\paragraph{Code Availability}

The \textsc{Matlab} source code for the numerical examples can be obtained from the doi \href{https://zenodo.org/doi/10.5281/zenodo.10059715}{10.5281/zenodo.10059715}.

\paragraph{Acknowledgments}

I thank Riccardo Morandin and Volker Mehrmann for helpful discussions.

\bibliographystyle{plain}
\bibliography{Refs}

\begin{thebibliography}{10}

\bibitem{Arn89}
V.~I. Arnold.
\newblock {\em Mathematical Methods of Classical Mechanics}.
\newblock Springer New York, USA, second edition, 1989.

\bibitem{CelH17}
E.~Celledoni and E.~H. H{\o}iseth.
\newblock Energy-preserving and passivity-consistent numerical discretization
  of port-{H}amiltonian systems.
\newblock ArXiv preprint 1706.08621v1, 2017.

\bibitem{DeuB02}
P.~Deuflhard and F.~Bornemann.
\newblock {\em Scientific Computing with Ordinary Differential Equations}.
\newblock Springer-Verlag New York, NY, USA, 2002.

\bibitem{EggHS21}
H.~Egger, O.~Habrich, and V.~Shashkov.
\newblock On the energy stable approximation of {H}amiltonian and gradient
  systems.
\newblock {\em J. Comput. Methods Appl. Math.}, 21(2):335--349, 2021.

\bibitem{FroGLM23}
A.~Frommer, M.~G\"unther, B.~Liljegren-Sailer, and N.~Marheineke.
\newblock Operator splitting for port-{H}amiltonian systems.
\newblock ArXiv preprint 2304.01766v1, 2023.

\bibitem{Gon96}
O.~Gonzalez.
\newblock Time integration and discrete {H}amiltonian systems.
\newblock {\em J. Nonlinear Sci.}, 6:449--467, 1996.

\bibitem{HaiLW06}
E.~Hairer, C.~Lubich, and G.~Wanner.
\newblock {\em Geometric Numerical Integration}.
\newblock Springer Berlin Heidelberg, Germany, second edition, 2006.

\bibitem{HaiW96}
E.~Hairer and G.~Wanner.
\newblock {\em Solving Ordinary Differential Equations II: Stiff and
  Differential-Algebraic Problems}.
\newblock Springer Berlin Heidelberg, Germany, second edition, 1996.

\bibitem{HirS74}
M.~W. Hirsch and S.~Smale.
\newblock {\em Differential Equations, Dynamical Systems, and Linear Algebra}.
\newblock Academic Press, New York, NY, USA, 1974.

\bibitem{JueST19}
A.~J\"ungel, U.~Stefanelli, and L.~Trussardi.
\newblock Two structure-preserving time discretizations for gradient flows.
\newblock {\em Appl. Math. Optim.}, 80:733--764, 2019.

\bibitem{KotL19}
P.~Kotyczka and L.~Lef\`{e}vre.
\newblock Discrete-time port-{H}amiltonian systems: a definition based on
  symplectic integration.
\newblock {\em Systems Control Lett.}, 133:104530, 2019.

\bibitem{KunM23}
P.~Kunkel and V.~Mehrmann.
\newblock Discretization of inherent {ODE}s and the geometric integration of
  {DAE}s with symmetries.
\newblock {\em BIT Numer. Math.}, 63:29, 2023.

\bibitem{McLQR99}
R.~I. McLachlan, G.~R.~W. Quispel, and N.~Robidoux.
\newblock Geometric integration using discrete gradients.
\newblock {\em Phil{.} Trans{.} R{.} Soc{.} Lond{.}}, 357(1754):1021--1045,
  1999.

\bibitem{MehM19}
V.~Mehrmann and R.~Morandin.
\newblock Structure-preserving discretization for port-{H}amiltonian descriptor
  systems.
\newblock In {\em Proceedings of the 58th IEEE Conference on Decision and
  Control}, pages 6863--6868, Nice, France, 2019.

\bibitem{Oet18}
H.~C. \"Ottinger.
\newblock {GENERIC} integrators: structure preserving time integration for
  thermodynamic systems.
\newblock {\em J. Non-Equil. Thermody.}, 43(2):89--100, 2018.

\bibitem{Sch23a}
P.~Schulze.
\newblock Structure-preserving model reduction for port-{H}amiltonian systems
  based on separable nonlinear approximation ansatzes.
\newblock {\em Front. Appl. Math. Stat.}, 9:1160250, 2023.

\bibitem{SchJ14}
A.~van~der Schaft and D.~Jeltsema.
\newblock {\em Port-{H}amiltonian Systems Theory: {A}n Introductory Overview}.
\newblock now Publishers Inc., Hanover, MA, USA, 2014.

\end{thebibliography}

\end{document}